\documentclass{amsart}
\usepackage[utf8]{inputenc}
\usepackage{amsthm,mathtools,amssymb,physics}
\usepackage{overpic,hyperref,cleveref}
\newtheorem{defi}{Definition}
\crefname{defi}{Definition}{}

\crefname{ass}{Assumption}{}
\newtheorem{ex}{Example}
\crefname{ex}{Example}{}
\newtheorem{lem}{Lemma}
\crefname{lem}{Lemma}{Lemmas}
\newtheorem{theo}{Theorem}
\crefname{theo}{Theorem}{Theorems}
\newtheorem{prop}{Proposition}
\crefname{prop}{proposition}{Propositions}
\newtheorem{rem}{Remark}
\crefname{rem}{Remarks}{}
\newtheorem{cor}{Corollary}
\crefname{cor}{Corollary}{}
\newcommand{\ep}{\varepsilon}
\newcommand{\la}{\Delta}
\renewcommand{\a}{\mathfrak{a}}
\newcommand{\A}{\mathcal{A}}

\newcommand{\Hil}{\mathcal{H}}
\newcommand{\N}{\mathbb{N}}
\renewcommand{\O}{\mathcal{O}}
\newcommand{\J}{\mathcal{J}}
\newcommand{\R}{\mathbb{R}}
\newcommand{\Z}{\mathbb{Z}}
\newcommand{\sq}{\square}
\newcommand{\sqi}{\sq_{i\ep}}
\newcommand{\Ki}{K_{i\ep}}
\newcommand{\Yi}{Y_{i\ep}}
\newcommand{\Hi}{H_{i\ep}}
\newcommand{\PI}{P_{i\ep}}
\newcommand{\QI}{Q_{i\ep}}
\newcommand{\di}{d_{i\ep}}
\newcommand{\Fi}{f_{i\ep}}
\newcommand{\ui}{u_{i\ep}}
\newcommand{\Vi}{V_{i\ep}}
\newcommand{\xI}{x_{i\ep}}
\newcommand{\chii}{\chi_{i\ep}}
\newcommand{\hchi}{\widehat{\chi}_{i\ep}}
\DeclareMathOperator{\capa}{cap}
\DeclareMathOperator{\diam}{diam}
\DeclareMathOperator{\dist}{dist}
\DeclareMathOperator{\spec}{spec}
\DeclareMathOperator{\Dom}{Dom}
\DeclareMathOperator{\supp}{supp}

\newcommand{\cls}[1]{\overline{#1}}
\newcommand{\pd}{\partial}

\begin{document}
\title[]{Convergence rate of Dirichlet Laplacians on domains with holes to the Schr\"{o}dinger operator with $L^p$ potential}
\author[H. Ishida]{Hiroto Ishida}
\address{Hiroto Ishida\\Graduate School of Science, University of Hyogo\\Shosha, Himeji, Hyogo 671-2201, Japan}
\email{immmrfff@gmail.com}
\subjclass[2020]{35B27, 47A10.}
\keywords{Homogenization, norm resolvent convergence, operator estimates}
\date{\today}
\maketitle
\begin{abstract}
We consider the Dirichlet Laplacian $\A_\ep=-\la$ in the domain $\Omega\setminus\bigcup_i \Ki\subset\R^n$ with holes $\Ki$ and the Schr\"{o}dinger operator $\A=-\la+V$ in $\Omega$ where $V$ is the $L^n(\Omega)$ limit of the density of the capacities $\capa(\Ki).$ Strong resolvent convergence for many $V\in W^{-1,\infty}(\Omega)$ was studied by the author. In this paper, we study about convergence rate for $\A_\ep\to\A$ in norm resolvent sense. The case for which $V$ is a constant is studied by Andrii Khrabustovskyi and Olaf Post.
\end{abstract}
\section{Introduction}
Let $\Omega\subset\R^n (n\geq 3)$ be a domain and $V\in L^n(\Omega,[0,\infty)).$ We consider the family of compact sets $\{\Ki\}_{i\in\Lambda_\ep}$ for each $\ep>0$ such that the density of the Newtonian capacities $\capa(\Ki)$ tends to $V$ in $L^n(\Omega)$ (see \Cref{fig} and \Cref{holes}). 
We let $\Omega_\ep=\Omega\setminus\bigcup_{i\in\Lambda_\ep}\Ki$ and consider the Dirichlet Laplacian $\A_\ep=-\la$ on $L^2(\Omega_\ep)$ and the Schr\"{o}dinger operator $\A=-\la+V$ on $L^2(\Omega).$ In this paper, we study about convergence rate of $\A_\ep\to\A$ in norm resolvent sense (see \Cref{NRC} for details for the statement).
Hereafter, we omit to denote "$\ep\to 0$" for convergence as $\ep\to 0.$
\subsection{Known results}
Resolvent convergence of Dirichlet Laplacian $\A_\ep$ on domains with holes is considered as homogenization problem for Poisson problem. The limit $V$ of the density of the capacities of the holes is often used to characterize the limit of $\A_\ep$ by $\A=-\la+V$ if each hole is a small compact set.
Convergence rate for norm resolvent convergence for $\A_\ep$ with holes is studied at \cite{CICE} when $V$ is a constant.
On the other hand, existence of the holes for which $\A_\ep\to\A$ in the strong resolvent sense for given $V\in W^{-1,\infty}(\Omega)$ is studied in \cite{INV}. \Cref{NRC} generalizes these results to norm resolvent convergence for $V\in L^n(\Omega).$

See \cite{KM,CM}, for homogenization for holes which may not be a union of compact sets.

This paper is organized as follows. We state the main result in \Cref{mainresu} under the assumptions in \Cref{secass}. We remark on properties of operators in \Cref{op}. Finally, we show the main result in \Cref{secconvrate} and \Cref{secproof}.
\section{Assumption and main results}\label{secass}
\subsection{Assumptions for the domain}
We assume $\Omega$ is bounded and $\pd\Omega$ is $C^2$ boundary. We also assume there exists $\theta>0$ such that the map $\pd\Omega\cross[0,\theta]\ni(x,t)\mapsto x+t\nu(x)\in\Omega$ is injective, where $\nu(x)$ is the unit-inward-pointing normal vector at $x\in\pd\Omega.$
\begin{rem}
These assumptions are required by \eqref{ellreg}(elliptic regularity),\cite[Lemma 4.7.]{CICE} and \Cref{compact} only. See \cite[Remark 4.8.]{CICE} to relax these assumptions.
\end{rem}
\subsection{Assumptions and construction of holes}\label{holes}
\begin{figure}\centering
\begin{overpic}[width=10cm]{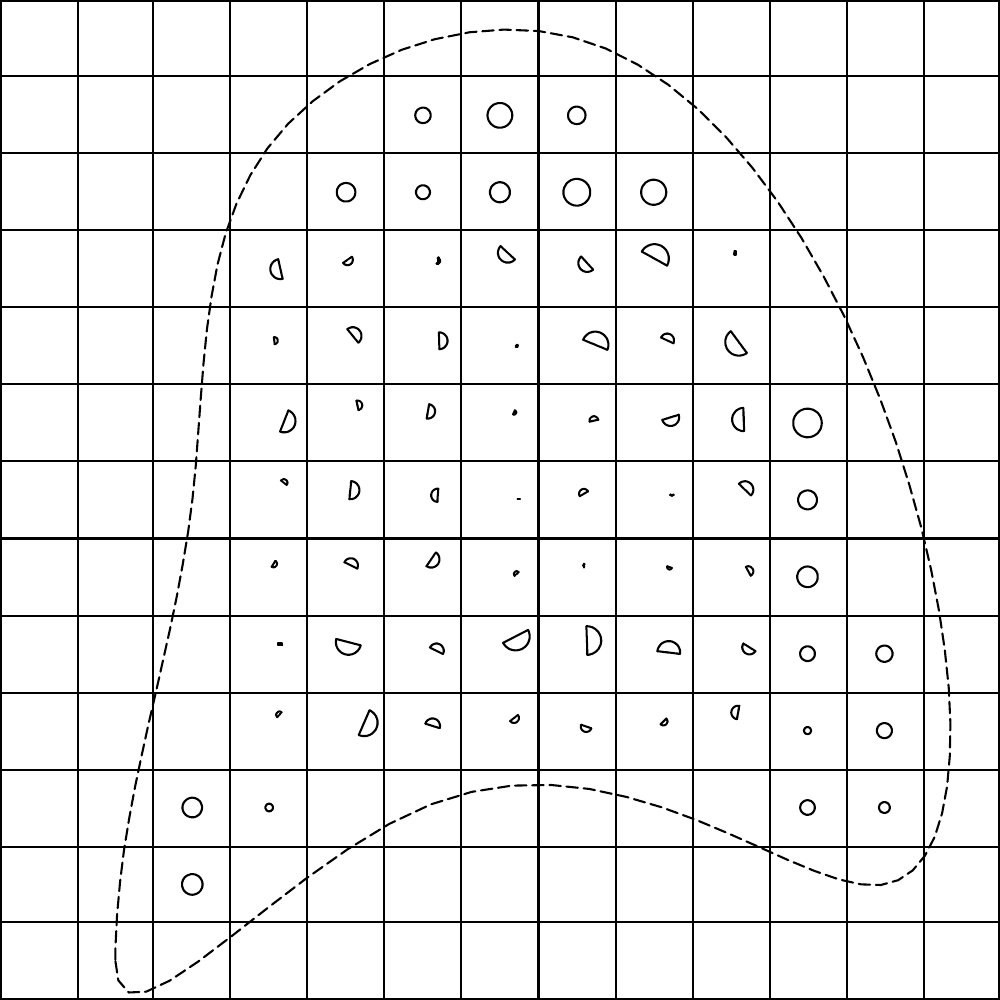}
\end{overpic}
\caption{A domains perforated by holes\label{fig}}
\end{figure}
Let $\sq=(-1/2,1/2)^n$ and $\sqi=\ep(\sq+i)$ for $\ep>0$ and $i\in\Z^n.$
We let
\[\Lambda^\ep=\{i\in\Z^n\mid \sqi\subset\Omega\}.\]
We assume $\Ki=\emptyset$ or $\pd\Ki$ is piecewise $C^1$ ("for each $0<\ep\ll1$ and $i\in\Lambda^\ep$" is omitted in this paper).
We recall definition of the Newtonian capacity. We define $\capa(\Ki)=\norm{\grad \Hi}_{L^2(\R^d)}^2,$ where $\Hi\in C(\R^n)$ is the solution to
\begin{gather}
\label{harmonic}
\la\Hi=0\mbox{ on }\R^n\setminus \Ki,\\
\nonumber
\Hi=1\mbox{ on }\Ki,\\
\nonumber
\Hi(x)\to0~(|x|\to\infty).
\end{gather}
We denote $B(x,r)=\{y\in\R^n\mid|x-y|<r\}$ for $x\in\R^n$ and $r>0.$ Let $\cls{B(\xI,\di)}$ be a smallest closed ball such that $\Ki\subset\cls{B(\xI,\di)}$ ($\di=0$ if 
 $\Ki=\emptyset$) and assume
\begin{equation}\label{asscap}
\capa(\Ki)=\int_{\sqi}V(x)dx.
\end{equation}
We assume that there exist $C,\kappa>0$ such that
\begin{equation}\label{assrad}
\di^{n-2}\leq C\capa(\Ki)
\end{equation}
and
\begin{equation}\label{assdist}
\dist(\cls{B(\xI,\di)},\pd\sqi)\geq\kappa\ep.
\end{equation}
We let
\[\Omega_\ep=\Omega\setminus\bigcup_{i\in\Lambda^\ep}\Ki.\]
We introduce a simple example satisfying these assumptions.
\begin{ex}
$K=\cls{B(0,1)}$ and $\Ki=\cls{B\qty(\ep i,(\frac{\int_{\sqi}V(x)dx}{\capa(K)})^\frac{1}{n-2})}$ satisfy the assumptions.
\begin{proof}
We verify \eqref{assdist} with $\xI=\ep i,~\di=(\frac{\int_{\sqi}V(x)dx}{\capa(K)})^\frac{1}{n-2}.$
Since $\capa(\cls{B(x,r)})=\capa(K)r^{n-1}$ for each $x\in\R^n$ and $r>0,$ we can verify \eqref{asscap} and \eqref{assrad}. 
H\"{o}lder's inequality implies $\capa(\Ki)\leq\norm{V}_{L^n(\Omega)}\ep^{n-1}.$ Therefore, we have $\sup_i\di/\ep\to 0.$ It implies \eqref{assdist}.
\end{proof}
\end{ex}
We define sesquilinear(quadratic) forms $\a_\ep$ on $L^2(\Omega_\ep)$ and $\a$ on $L^2(\Omega)$ by
\begin{align*}
\a_\ep(u,v)&=(\grad u,\grad v)_{L^2(\Omega_\ep)}~&(u,v\in\Dom(\a_\ep)&=H_0^1(\Omega_\ep))\\
\mbox{and}\\
\a(u,v)&=(\grad u,\grad v)_{L^2(\Omega)}+\int_\Omega u\cls{v}Vdx~&(u,v\in\Dom(\a)&=H_0^1(\Omega)).\end{align*}
We let $\A_\ep$ and $\A$ be self-adjoint operators associated with $\a_\ep$ and $\a$, respectively. We remark about definition of them at \Cref{op}.
\subsection{Main result}\label{mainresu}
We denote Lebesgue measure of $E\subset\R^n$ by $|E|.$ We denote
\[\Fi=\frac{\int_{\sqi}f(x)dx}{|\sqi|}~(f\in L^1(\sqi)\mbox{ or the zero extension of }f\in L^1(\sqi\setminus\Ki)).\]
For $V\in L^p(\Omega)~(n\leq p\leq\infty)$ and $0<\beta<1/2,$ let
\begin{gather*}
\gamma_n=\begin{cases}
1/2&(n=3)\\
1-\beta&(n=4)\\
1&(n\geq5)
\end{cases},
~e_\ep=\ep^\frac{(2-n/p)\gamma_n}{n-2}+\ep^{1-n/p},\\
b_\ep'=\sup_{\substack{E\neq\emptyset:\mbox{Borel set}\\\diam E\leq\diam\sq_{0\ep}}}\norm{V}_{L^p(E)},~b_\ep=(b_\ep')^\frac{1}{2(n-2)}+\ep^{n/p}\\
\mbox{and}\\
D_\ep=\sum_{i\in\Lambda^\ep}\norm{\Vi-V}_{L^n(\sqi)},
\end{gather*}
where $1/p=0$ if $p=\infty.$
\begin{prop}\label{convrate}
We have
\begin{gather}
\label{ldc}
D_\ep=\begin{cases}
\O(\ep)&(V\in W^{1,n}(\Omega))\\
\O(\ep^\alpha)&(V\in C^{0,\alpha}(\Omega),~\alpha>0)\\
o(1)&(V\in L^n(\Omega))
\end{cases}\\
\nonumber\mbox{and}\\
\label{intsmalldiam}
b_\ep=\begin{cases}
o(1)&(n\leq p<\infty)\\
\O(1)&(p=\infty)
\end{cases}.
\end{gather}
\end{prop}
\begin{rem}
We have $D_\ep+b_\ep e_\ep=o(1)$ even for $p=n.$
$\Vi=\capa(\Ki)/|\sqi|$ is the density of capacity $\capa(\Ki)$ by \eqref{asscap}. It means $V$ is the $L^n$ limit of the density of the capacities $\capa(\Ki).$
\end{rem}
We let $J_\ep\colon L^2(\Omega)\to L^2(\Omega_\ep)$ be the restriction operator and $J_\ep'\colon L^2(\Omega_\ep)\to L^2(\Omega)$ be the zero extension operator.
Now we state our main result.
\begin{theo}\label{deltacl}
$(L^2(\Omega),\a)$ and $(L^2(\Omega_\ep),\a_\ep)$ are $C(D_\ep+b_\ep e_\ep)$-close of order $2$ with respect to $J_\ep,~J_\ep',~J_\ep^1,~J_\ep^{1'}$ (see \Cref{app} in \Cref{appsec}, or \cite[Definition 3.1]{CICE}) for some $J_\ep^1,~J_\ep^{1'}$ and $C>0$ for $\ep\ll1.$
\end{theo}
We pick up important results given by \Cref{deltacl}.
\begin{theo}[Norm resolvent convergence {\cite[Corollary 3.8.]{CICE}}]\label{NRC}
\[\norm{J_\ep'(\A_\ep+1)^{-1}J_\ep-(\A+1)^{-1}}_{L^2(\Omega)\to L^2(\Omega)}=\O(D_\ep+b_\ep e_\ep).\]
\end{theo}
\begin{theo}\label{specconv}
$\spec(\A_\ep)$ and $\spec(\A)$ are discrete and consist of eigenvalues with finite multiplicity. 
Let $\lambda_{k,\ep}$ and $\lambda_k$ be $k$-th eigenvalue of $\A_\ep$ and $\A$ arranged in the ascending order and repeated according to their multiplicities. Then, we have
\[\sup_{k\in\N}\abs{(\lambda_{k,\ep}+1)^{-1}-(\lambda_k+1)^{-1}}
=\O(D_\ep+b_\ep e_\ep).\]
\end{theo}
\Cref{deltacl} also implies convergence for a function of $\A_\ep$ (see \cite[Theorem 3.7, 3.9]{CICE}, \cite[Appendix]{POST}).
\subsection{Outline of the proof}
The method of the proof of \Cref{deltacl} is similar with the proof of the results in \cite{CICE}. However, we need to change the method of the proof of \eqref{c5}. We use the fact that $\Dom(\A)$ has a property similar with elliptic regularity (\Cref{ellregLd}). Moreover, we arrange the method of estimation of $J_{\ep,2}$ in \cite{CICE} as seen in \eqref{I1} since $\capa(\Ki)$ depends on $i.$
\section{Remarks on operators}\label{op}
We need the lemma below to define $\A_\ep$ and $\A$ as self-operators associated with forms $\a_\ep$ and $\a.$ 
In this section, we denote $a(f,u)\lesssim b(f,u)$ if there exists $C>0$ which only depends on $\Omega$ and $V$ such that $a(f,u) \leq C b(f,u)$ for $a(f,u),~b(f,u)\in\R.$
\begin{lem}
$\a$ is a inner product on $\Dom(\a)$ which defines the norm equivalent with the norm of $H_0^1(\Omega).$ In particular, $\a$ is a closed symmetric form.
\begin{proof}
Sobolev's embedding theorem $W_0^{1,1}(\Omega)\hookrightarrow L^{(1-1/n)^{-1}}(\Omega)$ implies
\begin{equation}\label{w-1inf}
V\in L^n(\Omega)\hookrightarrow W^{-1,\infty}(\Omega).
\end{equation}
Let $u\in H_0^1(\Omega).$
Poincar\'{e}'s inequality and Schwarz's inequality implies $|\a(u,u)|\leq\norm{\grad u}_{L^2(\Omega)}^2+\norm{V}_{W^{-1,\infty}(\Omega)}$
$\norm{|u|^2}_{W_0^{1,1}(\Omega)}$
$\lesssim\norm{u}_{H_0^1(\Omega)}^2+\norm{u\grad u}_{L^1(\Omega)}\lesssim\norm{u}_{H_0^1(\Omega)}^2\lesssim\a(u,u).$
Therefore, the assertion follows.
\end{proof}
\end{lem}
We state some properties of $\A$ required to show the main results.
\begin{lem}\label{cptop}
$0\not\in\spec(\A)$ and $\norm{\A^{-1}}_{L^2(\Omega)\to H_0^1(\Omega)}<\infty.$
\begin{proof}
For any $f\in L^2(\Omega)\hookrightarrow H^{-1}(\Omega),$ Riesz representation theorem gives unique $u\in\Dom(\a)$ such that $(\A u,v)_{L^2(\Omega)}=\a(u,v)=(f,v)_{L^2(\Omega)}$ for any $v\in\Dom(\a).$ Clearly, $u=\A^{-1}f.$ Poincar\'{e}'s inequality imply
$\norm{u}_{L^2(\Omega)}^2\lesssim \norm{\grad u}_{L^2(\Omega)}^2\leq \a(u,u)$
$\leq \norm{f}_{L^2(\Omega)}\norm{u}_{L^2(\Omega)}.$ Therefore, we have $\norm{u}_{L^2(\Omega)}\lesssim \norm{f}_{L^2(\Omega)}.$ These inequalities imply $\norm{\grad u}_{L^2(\Omega)}\lesssim\norm{f}_{L^2(\Omega)}.$ Therefore, we have the assertion.
\end{proof}
\end{lem}
\begin{cor}\label{compact}
$\A^{-1}\colon L^2(\Omega)\to L^2(\Omega)$ is compact.
\begin{proof}
Let $f\to 0$ weakly in $L^2(\Omega).$ \Cref{cptop} implies $\A^{-1}f\to 0$ weakly in $H_0^1(\Omega).$ Rellich's theorem implies $\A^{-1}f\to 0$ in $L^2(\Omega).$ Therefore, the assertion follows.
\end{proof}
\end{cor}
\begin{rem}
We showed these results using \eqref{w-1inf} instead of $V\in L^n(\Omega)$ directly. Indeed, these results are still true if definition of $\a$ is replaced with
\[\a(u,v)=(\grad u,\grad v)_{L^2(\Omega)}+\int_\Omega u\cls{v}dV~(u,v\in\Dom(\a)=H_0^1(\Omega))\] where $V\geq0$ is a measure on $\Omega$ such that $V\in W^{-1,\infty}(\Omega).$
However, we have an advantage to assume $V\in L^n(\Omega)$ as below.
\end{rem}
\begin{lem}\label{ellregLd}
We have
$\Dom(\A)=H_0^1(\Omega)\cap H^2(\Omega)$ and there exists $c,c'>0$ such that $\norm{u}_{H^2(\Omega)}\leq c\norm{\A u}_{L^2(\Omega)}\leq c'\norm{(\A+1)u}_{L^2(\Omega)}$ for any $u\in\Dom(\A).$
\begin{proof}
Let $-\la_D$ be Dirichlet Laplacian on $L^2(\Omega)$ ($\A$ with $V=0$).
We have
\begin{equation}\label{ellreg}
\norm{u}_{H^2(\Omega)}\lesssim\norm{\la u}_{L^2(\Omega)}\mbox{ for } u\in\Dom(-\la_D)=H_0^1(\Omega)\cap H^2(\Omega)
\end{equation}
by elliptic regularity (e.g. \cite[Theorem 8.12.]{GT}).
Since $H_0^1(\Omega)\hookrightarrow L^{(\frac{1}{2}-\frac{1}{n})^{-1}}(\Omega),$ we have $\norm{Vu}_{L^2(\Omega)}
\leq\norm{V}_{L^n(\Omega)}\norm{u}_{L^{(\frac{1}{2}-\frac{1}{n})^{-1}}(\Omega)}
\lesssim\norm{u}_{H_0^1(\Omega)}$ for any $u\in H_0^1(\Omega).$ Therefore, we have
$\Dom(\A)=\{u\in H_0^1(\Omega)\mid-\la u+Vu\in L^2(\Omega)\}
=\{u\in H_0^1(\Omega)\mid\la u\in L^2(\Omega)\}
=\Dom(-\la_D)=H_0^1(\Omega)\cap H^2(\Omega).$
Let $u\in\Dom(\A).$ \Cref{cptop}, \eqref{ellreg} and $H_0^1(\Omega)\hookrightarrow L^{(\frac{1}{2}-\frac{1}{n})^{-1}}(\Omega)$ imply $\norm{u}_{H^2(\Omega)}$
$\lesssim\norm{\A u-V\A^{-1}\A u}_{L^2(\Omega)}$
$\lesssim\norm{\A u}_{L^2(\Omega)}
\leq\norm{(\A+1)u}_{L^2(\Omega)}+\norm{(\A+1)^{-1}(\A+1)u}_{L^2(\Omega)}$
$\lesssim\norm{(A+1)u}_{L^2(\Omega)}.$
\end{proof}
\end{lem}
\section{Convergence rate}\label{secconvrate}
In this section, we show \Cref{convrate}. We denote $\sum_i=\sum_{i\in\Lambda^\ep}$ in this paper. Similar for $\bigcup_i$ and $\sup_i.$
\begin{lem}\label{wirt}
Let $1\leq q<\infty.$
Then, there exists $C>0$ such that
\[\norm{u-\ui}_{L^q(\sqi)}\leq C\ep\norm{\grad u}_{L^q(\sqi)}\]
for any $u\in W^{1,q}(\sqi),~i\in\Z^n$ and $\ep>0.$
\begin{proof}
Let $v(x)=u(\ep(x+i)).$
The assertion for $i=0,~\ep=1$ is known as Poincar\'{e}-Wirtinger inequality. Therefore, we have
$\norm{v-v_{01}}_{L^q(\sq)}^q\leq C^q\norm{v}_{L^q(\sq)}^q.$
Changing the variable, we have $v_{01}=\ui,~\norm{v-v_{01}}_{L^q(\sq)}^q
=\ep^{-n}\norm{u-\ui}_{L^q(\sqi)}^q$ and $\norm{\grad v}_{L^q(\sq)}^q
=\ep^{q-n}\norm{\grad v}_{L^q(\sqi)}^q.$
Therefore, the assertion follows.
\end{proof}
\end{lem}
We denote
$1_E(x)=\begin{cases}
1&(x\in E)\\
0&(x\notin E)\end{cases}$ for $E\subset\R^n,~x\in\R^n.$
\begin{proof}[Proof of \eqref{ldc}]
The first case follows from \Cref{wirt}. Let $V\in C^{0,\alpha}(\Omega).$ Integrating
$\abs{\Vi-V(x)}1_{\sqi}(x)
\leq(\int_{\sqi}|V(y)-V(x)|dy)/|\sqi|
\leq|V|_{C^{0,\alpha}}(\diam\sq_{0\ep})^\alpha,$
we have the second case:
\begin{equation}\label{ldclip}
D_\ep
\leq|V|_{C^{0,\alpha}}(\diam\square)^\alpha|\Omega|^{1/n}\ep^\alpha.\end{equation}
For $V\in L^n(\Omega),$ take $\{f_\delta\}_{\delta>0}\subset C^{0,1}(\Omega)$ such that $\lim_{\delta\to0}\norm{f_\delta-V}_{L^n(\Omega)}=0.$

Integrating $\abs{(V-f_\delta)_{i\ep}}
\leq|\sqi|^{-1/n}\norm{V-f_\delta}_{L^n(\sqi)},$
we have \[\sum_i\norm{(V-f_\delta)_{i\ep}}_{L^n(\sqi)}^n
\leq\norm{V-f_\delta}_{L^n(\Omega)}^n.\]
It and \eqref{ldclip} with $\alpha=1$ imply
$D_\ep\leq2\norm{V-f_\delta}_{L^n(\Omega)}+|f_\delta|_{C^{0,1}}\diam\square|\Omega|^{1/n}\ep$ for each $\ep>0$ and $\delta>0.$
Therefore we have $\limsup_{\ep\to 0}D_\ep=0.$
\end{proof}
We can not clarify convergence rate for $V\in L^n(\Omega)$ generally. We introduce a example as below.
\begin{ex}
$V=2\cross1_{[0,\infty]\cross\R^{n-1}}$ satisfy
$D_\ep=\O(\ep^{1/n}).$
\begin{proof}
Choose $R>0$ such that $[-R,R]^n\supset\Omega.$ Since
$\sum_i\Vi$
$=1_{(-\ep/2,\ep/2)\cross\R^{n-1}}$
$+2\cross 1_{(\ep/2,\infty)\cross\R^{n-1}},$ we have
$D_\ep\leq\norm{1}_{L^n((-\ep/2,-\ep/2)\cross[-R,R]^{n-1})}\leq((2R)^{n-1}\ep)^{1/n}.$
\end{proof}
\end{ex}
We need the lemma below to show \eqref{intsmalldiam}. It seems the lemma is very famous. But we show it for the sake.
\begin{lem}
Let $\mu$ be a Borel measure on $\R^n$ such that $\mu(\{x\})=0$ for any $x\in\R^n.$ Let $B\subset\R^n$ be a bounded Borel set such that $\mu(B)<\infty.$
Then, we have
\[\sup\{\mu(E)\mid E:\mbox{Borel set}\subset B,~\diam E<\ep\}\to0.\]
\begin{proof}
We show it only for $B\neq\emptyset.$
Take $E_\ep\subset B$ such that $\diam E_\ep<\ep,~E_\ep\neq\emptyset$ and
$\sup\{\mu(E)\mid E\neq\emptyset:\mbox{Borel set}\subset B,~\diam E<\ep\}-\ep<\mu(E_\ep).$
Let $t=\limsup_\ep\mu(E_\ep).$ Take a subsequence $\mu(E_{\ep_m})\to t~(\ep_m\to 0)$ such that there exists $x_m\in E_{\ep_m}$ converging to some $x\in\cls{B}$ and $|x-x_m|+\diam E_{\ep_m}<1/m.$
Since $E_{\ep_m}\subset B(x,1/m),$ we have $t=\lim_m\mu(E_{\ep_m})\leq\lim_m\mu(B(x,1/m))=\mu(\{x\})=0.$
\end{proof}
\end{lem}
\begin{proof}[Proof of \eqref{intsmalldiam}]
The assertion is shown by applying the lemma above for  $B=\Omega,~\mu(E)$
$=\int_{\Omega\cap E}|V|^pdx$ for Borel set $E$ of $\R^n.$
\end{proof}

\section{Proof of the main result}\label{secproof}
In this section, we show \Cref{deltacl} and \Cref{specconv}.
We take $\chi\in C^\infty(\R,[0,1])$ such that
$\chi(x)=\begin{cases}
1&(x<1)\\
0&(x>2)
\end{cases}.$
Let 
\[\hchi(x)=\chi\qty(\frac{2/\kappa(|x-\xI|-\di)}{\ep}),~\chii(x)=\chi\qty(\frac{|x-\xI|}{\di})\]
(see also \cite[Fig. 3.]{CICE}) and define $J_\ep^1\colon\Dom(\a)\to\Dom(\a_\ep)$ by 
\[\PI f=(f-\Fi)\chii,~\QI f=\Fi \Hi\hchi,~J_\ep^1f=f-\sum_i(\PI+\QI)f~(f\in\Dom(\a)).\]
Let $J^{1'}_\ep\colon\Dom(\a_\ep)
=H_0^1(\Omega_\ep)\to\Dom(\a)=H_0^1(\Omega)$ be the zero extension operator.
Now we show \Cref{deltacl} under our settings.
\eqref{c1b},\eqref{c2},\eqref{c3a},\eqref{c3b},\eqref{c4a} with $\delta_\ep=0$ for our situation follows as noted in \cite[page 151]{CICE}. Therefore, it is enough to verify \eqref{c1a},\eqref{c4a} and \eqref{c5}.

We denote  $a_i^\ep(f,u)\lesssim b_i^\ep(f,u)$ if there exists $C>0$ not depending on $i\in\Lambda_\ep,\ep\ll1,u$ and $f$ such that $a_i^\ep(f,u)\leq C b_i^\ep(f,u)$ for $a_i^\ep(f,u),~b_i^\ep(f,u)\in\R.$ We regard $1/p=0$ if $p=\infty.$
\begin{lem}[{\cite[Lemma 2.1]{CICE},\cite[Lemma 2.4]{KM}}]\label{estpot}
There exists $c'>0$ such that
\[
|\pd^\alpha\Hi(x)|
\leq c'\di^{n-2}(|x-\xI|-\di)^{-n+2-|\alpha|}~(|x-\xI|-\di\geq c'\di)\]
for each $\ep>0,~i\in\Lambda^\ep$ and $|\alpha|\leq 1.$
\end{lem}
\begin{proof}[Proof of \eqref{c1a}]
It is enough to show 
\[\norm{J_\ep f-J_\ep^1f}_{L^2(\Omega_\ep)}
\lesssim(\ep+b_\ep e_\ep)\norm{ f}_{H^1(\Omega)}~(f\in H_0^1(\Omega))\]
since $\ep=\ep^{1-n/p}\ep^{n/p}\leq e_\ep b_\ep$ and
\begin{equation}\label{equiv}
\norm{f}_{H^1(\Omega)}^2
\leq\a(f,f)+(f,f)_{L^2(\Omega)}.
\end{equation}

It is enough to show $\sum_i\norm{\QI f}_{L^2(\sqi)}^2\lesssim(b_\ep e_\ep)^2\norm{f}_{H^1(\Omega)}^2$ by \cite[(28),~(29)]{CICE}.

Applying H\"{o}lder's inequality for \eqref{asscap} and \eqref{assrad}, we have
\begin{equation}\label{ordcap}
\di^{n-2},~\capa(\Ki)\lesssim b_\ep'\ep^{n-n/p}.
\end{equation}
Integrating \Cref{estpot} on $\{x\in\R^n\mid\kappa\ep/2\leq|x-\xI|-\di\leq\kappa\ep\}\supset\supp\grad\hchi$, and using \eqref{ordcap}, we have
\begin{equation}\label{estpotint}
\norm{\Hi}_{L^2(\supp\grad\hchi)}^2
\lesssim\di^{2(n-2)}\ep^{-n+4}
\lesssim (b_\ep')^2
\ep^{n+4-2n/p}.
\end{equation}
Schwarz's inequality implies
\begin{equation}\label{schmean}
|\Fi|\leq\ep^{-n/2}\norm{f}_{L^2(\sqi)}.
\end{equation}
\cite[(30)]{CICE}, \eqref{schmean}, \eqref{ordcap},~\eqref{estpotint} and $(b_\ep')^2\leq b_\ep'\leq b_\ep^2$ imply
\begin{align*}
\sum_i\norm{\QI f}_{L^2(\sqi)}^2
&\lesssim \norm{f}_{L^2(\Omega)}^2\ep^{2-n}\sup_i\qty(\capa(\Ki)+\ep^{-2}\norm{\Hi}_{L^2(\supp\grad\hchi)}^2)\\
&\lesssim\norm{f}_{L^2(\Omega)}^2b_\ep^2(\ep^{2-n/p}+\ep^{4-2n/p})
\lesssim\norm{f}_{H^1(\Omega)}^2(b_\ep e_\ep)^2.
\end{align*}
\end{proof}
\begin{proof}[Proof of \eqref{c4a}]
It is enough to show
\[\norm{f}_{L^2(\bigcup_i\Ki)}
\lesssim b_\ep e_\ep\norm{f}_{H^1(\Omega)}~(f\in H_0^1(\Omega))\]
since \eqref{equiv}.
\eqref{ordcap} implies
\begin{equation}\label{ratedep}
\di/\ep\lesssim b_\ep^2\ep^\frac{2-n/p}{n-2}
\end{equation}
and
\begin{equation}\label{spseq}
(\di/\ep)^n+\di\ep\lesssim b_\ep^2e_\ep^2.
\end{equation}
\cite[Lemma 4.2.]{CICE} (shown in \cite[Lemma 4.9.,Remark. 4.2]{KM}) with $D=D_2=\sqi,~D_1=\cls{B(\xI,2\di)}$ and \eqref{spseq}, we have
\begin{equation}\label{nearlyc4a}
\norm{f}_{L^2(\bigcup_i\cls{B(\xI,2\di)})}^2
\lesssim(b_\ep e_\ep)^2\norm{\grad f}_{L^2(\Omega)}^2
\end{equation}
which implies the assertion.
\end{proof}
\begin{proof}[Proof of \eqref{c5}]
It is enough to show
\begin{align*}
&\abs{\a_\ep(J_\ep^1f,u)-\a(f,J_\ep^{1'} u)}\\
\lesssim&(D_\ep+b_\ep e_\ep)\norm{f}_{H^2(\Omega)}\norm{u}_{H_0^1(\Omega_\ep)}~(f\in H_0^1(\Omega)\cap H^2(\Omega),~u\in H_0^1(\Omega_\ep))
\end{align*}
since \Cref{ellregLd} and $\norm{u}_{H_0^1(\Omega_\ep)}^2\leq\a_\ep(u,u)+(u,u)_{L^2(\Omega_\ep)}.$
We denote $\Yi=\sqi\setminus\Ki$ and $T_\ep=\Omega_\ep\setminus\cls{\bigcup_i\Yi}.$
Similarly for \cite[4.3.]{CICE}, we have
\begin{align*}
&\abs{\a_\ep(J_\ep^1f,u)-\a(f,J_\ep^{1'} u)}\\
\leq&\abs{\sum_i(\grad u,\grad\PI f)_{L^2(\Yi)}}+\abs{\sum_i\qty((\grad u,\grad\QI f)_{L^2(\Yi)}+\int_{\sqi}Vf\cls{u}dx)}+\abs{\int_{T_\ep}Vf\cls{u}dx}\\
\eqqcolon&\J_{\ep,1}+\J_{\ep,2}+\J_{\ep,3}.
\end{align*}
\cite[Lemma 4.7]{CICE}, $H_0^1(\Omega)\hookrightarrow L^{(\frac{1}{2}-\frac{1}{n})^{-1}}$ and H\"{o}lder's inequality for $\frac{1}{n}+(\frac{1}{2}-\frac{1}{n})+\frac{1}{2}=1$ imply
\[\J_{\ep,3}\lesssim\ep\norm{f}_{H_0^1(\Omega)}\norm{u}_{H_0^1(\Omega_\ep)}
\leq b_\ep e_\ep\norm{f}_{H^2(\Omega)}\norm{u}_{H_0^1(\Omega_\ep)}.\]
\eqref{ratedep} and similarly for proof of \cite[(46)]{CICE}, we have
\begin{equation}\label{j1pre}
\sqrt{\sum_i\norm{(f-\Fi)\chii}_{L^2(\Yi)}^2}
\lesssim b_\ep^{2\gamma_\ep}e_\ep\norm{f}_{H^2(\Omega)}
\lesssim b_\ep e_\ep\norm{f}_{H^2(\Omega)}.
\end{equation}
Replacing $f$ of \eqref{nearlyc4a} with $ \pd_jf~(j\leq n)$ and applying $\supp\chii\subset\cls{B(\xI,2\di)},$ we have
\begin{equation}\label{j2pre}
\sum_i\norm{\grad f}_{L^2(\supp\chii)}^2\lesssim(b_\ep e_\ep)^2\norm{f}_{H^2(\Omega)}^2.
\end{equation}
\cite[(36)]{CICE}, \eqref{j1pre} and \eqref{j2pre} imply
\[\J_{\ep,1}
\lesssim b_\ep e_\ep\norm{u}_{H_0^1(\Omega_\ep)}
\norm{f}_{H^2(\Omega)}.\]
\cite[(52)]{CICE} implies
\begin{align*}
\J_{\ep,2}
&\leq\sum_i\abs{\Fi \cls{\ui}\capa(\Ki)-\int_{\sqi}Vf\cls{u}dx}+\sum_i|\Fi|\norm{\ui-u}_{L^2(\Yi)}\norm{\la(\Hi\hchi)}_{L^2(\Yi)}\\
&\eqqcolon I_1+I_2.
\end{align*}
Let $V^\ep=\sum_i\Vi1_{\sqi}.$
H\"{o}lder's inequality for $\frac{1}{n}+(\frac{1}{2}-\frac{1}{n})+\frac{1}{2}=1,$
$~\norm{1}_{L^2(\sqi)}=\ep^{n/2},~H_0^1(\Omega)\hookrightarrow L^{(1/2-1/n)^{-1}},~\ui\capa(\Ki)=\int_{\sqi}uV^\ep dx,$ \Cref{wirt} and \eqref{schmean} imply
\begin{align}\begin{split}\label{I1}
I_1
&\leq\sum_i\qty(|\Fi|\int_{\sqi}|u||V^\ep-V|dx+\int_{\sqi}|f-\Fi||u|Vdx)\\
&\lesssim(D_\ep+\ep)\norm{f}_{H_0^1(\Omega)}\norm{u}_{H_0^1(\Omega_\ep)}
\lesssim(D_\ep+b_\ep e_\ep)\norm{f}_{H^2(\Omega)}\norm{u}_{H_0^1(\Omega_\ep)}
.
\end{split}\end{align}
\Cref{estpot},~\eqref{ordcap} and \eqref{harmonic} imply
$|\la(\Hi\hchi)|1_{\Yi}\lesssim \di^{n-2}\ep^{-n}\lesssim b_\ep'\ep^{-n/p}.$
It, \Cref{wirt} and \eqref{schmean} imply
$I_2\lesssim b_\ep'\ep^{1-n/p}\norm{f}_{L^2(\Omega)}\norm{u}_{H_0^1(\Omega_\ep)}
\lesssim b_\ep e_\ep\norm{f}_{H^2(\Omega)}\norm{u}_{H_0^1(\Omega_\ep)}.$
\end{proof}
Therefore, \Cref{deltacl} is shown.
\begin{proof}[Proof of \Cref{specconv}]
We first show
\begin{equation}\label{a1}\tag{$A_1$}
\norm{J_\ep f}_{L^2(\Omega_\ep)}\to\norm{f}_{L^2(\Omega)}~(f\in L^2(\Omega)).
\end{equation}
\eqref{spseq} and $\sum_i1\leq|\Omega|/|\sq_{0\ep}|$ imply $\norm{1_{\cup_i\Ki}}_{L^1(\Omega)}
\lesssim\sup_i(\di/\ep)^n\to 0.$
Let $t_\ep=\norm{f}_{L^2(\bigcup_i\Ki)},~t=\limsup_\ep t_\ep$ and take a subsequence such that $t_{\ep_k}\to t$ and $1_{\bigcup_i K_{i\ep_k}}\to 0$ a.e. as $k\to\infty.$
Dominated convergence theorem implies $t=0.$ Therefore, we have $\abs{\norm{J_\ep f}_{L^2(\Omega_\ep)}-\norm{f}_{L^2(\Omega)}}$
$\leq t_\ep\to0.$  Therefore, we have \eqref{a1}.

The assertion is shown by \Cref{compact}, \eqref{a1} and the method in \cite[4.4. Proof of Theorem 2.7.]{CICE}.
\end{proof}
\section{Open problems}
Convergence rate for $V\in L^n(\Omega)$ is no more than $o(1)$ in this paper. Moreover, we can not clarify convergence rate for $V\in W^{-1,\infty}(\Omega)$ by the method of the proof of \eqref{c5} in this paper.
Convergence rate of $\A_\ep\to\A$ in norm resolvent sense could be improved by devising construction of holes.
\section{Appendix}\label{appsec}
We recall definition of $\delta_\ep-$closeness of forms. In this section, $\ep$ does not mean a index.
Let $\Hil$ and $\Hil_\ep$ be separable Hilbert spaces. Let $\a,~\a_\ep\geq 0$ be closed, densely defined sesquilinear forms in $\Hil$ and $\Hil_\ep$ respectively. Let $\A$ and $\A_\ep$ be self-adjoint operators associated with $\a$ and $\a_\ep$ respectively. Let
\[\Hil^1=\Dom(\a),~\norm{u}_{\Hil^1}=\sqrt{\a(u,u)+\norm{u}_\Hil^2}~(u\in\Hil^1)\]
and define $\Hil_\ep^1$ by $\a_\ep$ similarly.
Let $J_\ep\colon\Hil\to\Hil_\ep,~J_\ep'\colon\Hil_\ep\to\Hil,$ $J_\ep^1\colon\Hil^1\to\Hil_\ep^1$ and  $J_\ep^{1'}\colon\Hil_\ep^1\to\Hil^1$ be linear bounded operators. Let $\delta_\ep>0.$
\begin{defi}[{\cite[Definition 3.1]{CICE}}]\label{app}
We say $(\Hil,\a)$ and $(\Hil_\ep,a_\ep)$ are $\delta_\ep-$close of order $2$ respect to $J_\ep,~J_\ep',~J_\ep^1,~J_\ep^{1'}$ if the following conditions hold:
\begin{gather}
\label{c1a}\tag{$C_{1a}$}
\norm{J_\ep-J_\ep^1}_{\Hil^1\to\Hil_\ep}
\leq\delta_\ep,\\
\label{c1b}\tag{$C_{1b}$}
\norm{J_\ep'-J_\ep^{1'}}_{\Hil_\ep^1\to\Hil}
\leq\delta_\ep,\\
\label{c2}\tag{$C_2$}
\abs{(J_\ep f,u)_{\Hil_\ep}-(f,J_\ep'u)_\Hil}
\leq\delta_\ep\norm{f}_\Hil\norm{u}_{\Hil_\ep}
~(f\in\Hil,~u\in\Hil_\ep),\\
\label{c3a}\tag{$C_{3a}$}
\norm{J_\ep}_{\Hil\to\Hil_\ep}
\leq1+\delta_\ep,\\
\label{c3b}\tag{$C_{3b}$}
\norm{J_\ep'}_{\Hil_\ep\to\Hil}
\leq1+\delta_\ep,\\
\label{c4a}\tag{$C_{4a}$}
\norm{1-J_\ep'J_\ep}_{\Hil^1\to\Hil}
\leq\delta_\ep,\\
\label{c4b}\tag{$C_{4b}$}
\norm{1-J_\ep J_\ep'}_{\Hil_\ep^1\to\Hil_\ep}
\leq\delta_\ep,\\
\label{c5}\tag{$C_5$}
\abs{\a_\ep(J_\ep^1f,u)-\a(f,J_\ep^{1'}u)}
\leq\delta_\ep\norm{(\A+1)f}_{\Hil}\norm{u}_{\Hil_\ep^1}
~(f\in\Dom(\A),~u\in\Hil_\ep^1).
\end{gather}
\end{defi}
\bibliographystyle{plain}
\bibliography{Cite/CM,Cite/CICE,Cite/KM,Cite/INV,Cite/POST,Cite/GT}
\end{document}